\documentclass[notitlepage, 11pt]{article}
\usepackage[shortlabels]{enumitem}
\usepackage{amsmath}
\usepackage{amssymb}
\usepackage{amsthm}
\usepackage{abstract}
\usepackage{xcolor}
\usepackage{comment}
\usepackage{mathtools}
\usepackage{graphicx}
\usepackage{caption}
\usepackage{subcaption}
\usepackage{float}
\usepackage{enumitem}
\setlist[itemize]{noitemsep}
\setlist[enumerate]{noitemsep}
\usepackage{hyperref}
\usepackage{cleveref}
\hypersetup{
    colorlinks=true,
    linkcolor=blue,
    filecolor=magenta,      
    urlcolor=cyan
    }
\makeatletter
\newcommand*{\barfix}[2][.175ex]{%
  \mathpalette{\@barfix{#1}}{#2}%
}
\newcommand*{\@barfix}[3]{%
  \vbox{%
    \kern#1\relax
    \hbox{$#2#3\m@th$}%
  }%
}
\makeatother

\newtheorem{theorem}{Theorem}
\newtheorem{thm}{Theorem}[section]
\newtheorem{corollary}[thm]{Corollary}
\newtheorem{lemma}[thm]{Lemma}

\newtheorem{claim}[thm]{Claim}

\newtheorem{question}[thm]{Question}
\newcommand{\footremember}[2]{%
    \footnote{#2}
    \newcounter{#1}
    \setcounter{#1}{\value{footnote}}%
}
\newcommand{\footrecall}[1]{%
    \footnotemark[\value{#1}]%
} 

\usepackage[margin=2cm]{geometry} 

\title{\vspace{-2em}Components, large and small, are as they should be I:\\
supercritical percolation on regular graphs of growing degree}
\author{%
Sahar Diskin \footremember{alley}{School of Mathematical Sciences, Tel Aviv University, Tel Aviv 6997801, Israel. \\Emails: sahardiskin@mail.tau.ac.il, krivelev@tauex.tau.ac.il.}%
\and Michael Krivelevich \footrecall{alley}%
}
\date{}

\begin{document}
\maketitle
\vspace{-2em}
\begin{abstract}
We provide sufficient conditions for a regular graph $G$ of growing degree $d$, guaranteeing a phase transition in its random subgraph $G_p$ similar to that of $G(n,p)$ when $p\cdot d\approx 1$. These conditions capture several well-studied graphs, such as (percolation on) the complete graph $K_n$, the binary hypercube $Q^d$, $d$-regular expanders, and random $d$-regular graphs. In particular, this serves as a unified proof for these (and other) cases.

Suppose that $G$ is a $d$-regular graph on $n$ vertices, with $d=\omega(1)$. Let $\epsilon>0$ be a small constant, and let $p=\frac{1+\epsilon}{d}$. Let $y(\epsilon)$ be the survival probability of a Galton-Watson tree with offspring distribution Po$(1+\epsilon)$. We show that if $G$ satisfies a (very) mild edge expansion requirement, and if one has fairly good control on the expansion of small sets in $G$, then typically the percolated random subgraph $G_p$ contains a unique giant component of asymptotic order $y(\epsilon)n$, and all the other components in $G_p$ are of order $O(\log n/\epsilon^2)$. 

We also show that this result is tight, in the sense that if one asks for a slightly weaker control on the expansion of small sets in $G$, then there are $d$-regular graphs $G$ on $n$ vertices, where typically the second largest component is of order $\Omega(d\log (n/d))=\omega(\log n)$. 

This is the first of a two-part sequence of papers. In the subsequent work, we consider supercritical percolation on regular graphs of \textit{constant} degree, and establish similar sufficient (and essentially tight) conditions in that setting.
\end{abstract}
\section{Introduction}
\subsection{Background and motivation}
Given a \textit{host graph} $G$ and a probability $p\in [0,1]$, the \textit{percolated random subgraph} $G_p\subseteq G$ is obtained by retaining each edge of $G$ independently with probability $p$. The study of this model was initiated by Broadbent and Hammersley in 1957 \cite{BH57}, where it was used to model the flow of fluid through a medium with randomly blocked channels. A particularly well-studied example of this model is the \textit{binomial random graph} $G(n,p)$, which is equivalent to percolation with probability $p$ on the complete graph $K_n$. For more background on random graphs and on percolation, see \cite{B01,BR06,FK16,G99,JLR00,K82}.

A classical result of Erd\H{o}s and R\'enyi from 1960 \cite{ER60} states that the binomial random graph $G(n,p)$ undergoes a fundamental phase transition with respect to its component structure when the expected average degree is around $1$ (that is, $p\cdot n\approx 1$). More precisely, given a constant $\epsilon>0$, let us define $y\coloneqq y(\epsilon)$ to be the unique solution in $(0,1)$ of the equation
\begin{align}\label{survival prob}
    y=1-\exp\left\{-(1+\epsilon)y\right\}.
\end{align}
Note that $y(\epsilon)$ is the survival probability of a Galton-Watson tree with offspring distribution Po$(1+\epsilon)$. Erd\H{o}s and R\'enyi then showed that when $p=\frac{1-\epsilon}{n}$, typically all components of $G(n,p)$ are of order $O(\log n/\epsilon^2)$, and when $p=\frac{1+\epsilon}{n}$, then typically there exists a unique giant component $L_1$ in $G(n,p)$ of order $(1+o(1))y(\epsilon)n$, and all other components of $G(n,p)$ are of order $O(\log n/\epsilon^2)$.

Note that in the \textit{supercritical regime}, when the expected average degree is larger than one, typically the second largest component has order $O\left(\frac{\log n}{\epsilon^2}\right)$. This is the same as the typical order of the largest components in the \textit{subcritical regime}, when the expected average degree is smaller than one. This is termed the \textit{duality principle} (see, for example, \cite{HH17}) --- in $G(n,p)$, the distribution of the components \textit{outside} $L_1$ in the supercritical regime, is the same as the distribution of the components in the subcritical regime. Throughout the rest of the paper, we say that a $d$-regular host graph $G$ exhibits the \textit{Erd\H{o}s-R\'enyi component phenomenon} (ERCP), if $G_p$ undergoes a phase transition similar to the above around $p=\frac{1}{d}$ --- that is, when $p=\frac{1+\epsilon}{d}$ a unique giant component of asymptotic order $y(\epsilon)|V(G)|$ emerges, and typically all other components have order $O_{\epsilon}\left(\log |V(G)|\right)$.

Several other host graphs have been shown to exhibit the ERCP: the binary hypercube $Q^d$~\cite{AKS81, BKL92}, products of many regular graphs \cite{DEKK22}, and pseudo-random $(n,d,\lambda)$-graphs \cite{FKM04}, to name a few. This suggests that there is some universality class, which `captures' the ERCP. Indeed, when the host graph $G$ is $d$-regular (in fact, when it has maximum degree $d$) on $n$ vertices, it is known (see, for example, \cite[Theorem 2]{DEKK22}) that when $p=\frac{1-\epsilon}{d}$ \textbf{whp}\footnote{With high probability, that is, with probability tending to one as $d$ tends to infinity. Throughout the paper, we treat $d$ as an asymptotic parameter tending to infinity.} all components of $G_p$ are of order $O\left(\frac{\log n}{\epsilon^2}\right)$. However, taking $G$ to be a disjoint union of cliques $K_{d+1}$, it is clear that when $p=\frac{1+\epsilon}{d}$, and in fact for any value of $p$, the largest component will have (deterministically) order at most $d+1$. Thus, some requirements on the edge distribution of the graph are necessary. 

In trying to characterise graphs which exhibit the ERCP, a first natural step is to ask for requirements on $G$ such that typically, under percolation, it exhibits the emergence of a giant component. For constant degree $d$, Alon, Benjamini, and Stacey \cite{ABS04} showed that if $G$ is an high-girth expander, then when $p(d-1)>1$, typically there exists a giant component in $G_p$. Subsequent work by Krivelevich, Lubetzky, and Sudakov~\cite{KLS20} showed that this giant component has asymptotic order $y(\epsilon)|V(G)|$, and yet for any $a<1$ there are high-girth expanders $G$ where, when $p(d-1)>1$, the second largest component is \textbf{whp} of order $|V(G)|^{a}$. Alimohammadi, Borgs, and Saberi \cite{ABS23} showed that for graphs $G$ with bounded average degree $d$, under the assumption that large sets (of linear order) expand, the typical asymptotic order of the giant component is dictated by the local structure of $G$. Van der Hofstad \cite{H23} gave a similar result for the typical asymptotic order of the giant, under the assumption that it is unlikely for two random vertices to lie in distinct large components, thus relating the density of the giant component to branching process approximations, and further showed that this assumption is essentially necessary as well. For growing degree $d$, Diskin, Erde, Kang, and Krivelevich \cite{DEKK25} showed that if the host graph $G$ is $d$-regular and satisfies that every $U\subseteq V(G)$ with $\Omega(|V(G)|)=|U|\le \frac{|V(G)|}{2}$ has that $e(U,U^C)\ge C|U|$ for some large enough constant $C\coloneqq C(\epsilon)$, then \textbf{whp}, when $p=\frac{1+\epsilon}{d}$ there exists a giant component in $G_p$ whose asymptotic order is $y(\epsilon)n$. In \cite{DEKK25}, it was further shown that this assumption alone does not suffice to show the full ERCP, that is, that the second largest component of $G_p$ is typically of logarithmic order. Let us also note here a recent result of Lichev, Mitsche, and Pernanau \cite{LMP24}, where they consider two sources of randomness: first, they choose a graph $G$ uniformly at random from all graphs with a given degree sequence, and then they percolate with probability $p$. They provide criteria characterising the behaviour of thresholds in that setting, and also give an example of a degree sequence for which the order of the largest component undergoes an unbounded number of jumps in terms of the percolation parameter. 
\subsection{Main results}
In this paper, we restrict our attention to $d$-regular graphs. Indeed, in \cite{CDE24,DEKK23}, it was shown that even if the host graph has a product structure, if we allow the graph to be quite irregular, there are examples where typically it will not exhibit the ERCP (see also the previous discussion on the work of \cite{LMP24}). Here, we show that asking for a very mild edge expansion together with some control on the `local' expansion of the graph provides a \textit{tight} characterisation of $d$-regular graphs $G$ exhibiting the ERCP. Our first main result considers graphs whose degree is at least poly-logarithmic in the number of vertices. We note that, throughout the paper, given a graph $G=(V,E)$ we denote by $N(A)$ the external neighbourhood of a set $A\subseteq V$ in $G$.  
\begin{theorem}\label{th: main}
Let $\alpha, c_1, c_2, c_3>0$ be constants and let $\epsilon>0$ be a sufficiently small constant. Suppose that $d$ and $n$ satisfy that $d\ge \log^{\alpha}n$. Let $G$ be a $d$-regular graph on $n$ vertices, and let $p=\frac{1+\epsilon}{d}$. Suppose that $G$ satisfies the following properties.
\begin{enumerate}[(P\arabic*)]
    \item For every $U\subseteq V(G)$ with $|U|\le \frac{n}{2}$, $e(U,U^C)\ge c_1|U|$. \label{p: global expansion}
    \item For every $U\subseteq V(G)$ with $|U|\le c_2\log n$, $|N(U)|\ge c_3d|U|$. \label{p: ball growth}
\end{enumerate}
Assume in addition that 
\begin{enumerate}[(P\arabic*)]
\setcounter{enumi}{2}
    \item for every $U\subseteq V(G)$ with $|U|\le Cd\log n$, $e(U,U^C)\ge (1-\epsilon^3)d|U|$, \label{p: small set expansion}
\end{enumerate}
for a sufficiently large constant $C\coloneqq C(\epsilon,c_1,c_2,c_3,\alpha)>0$. Then, \textbf{whp} there exists a unique giant component in $G_p$ whose order is $\left(1+o(1)\right)y(\epsilon)n$, where $y(\epsilon)$ is given by \eqref{survival prob}. Further, \textbf{whp} all the other components in $G_p$ are of order $O\left(\frac{\log n}{\epsilon^2}\right)$.
\end{theorem}
A few comments are in place. Note that besides the explicit assumption that $d\ge \log^{\alpha}n$ for some $\alpha>0$, Properties \ref{p: ball growth} and \ref{p: small set expansion} imply that $d=O(n/\log n)$. In subsequent theorems (Theorems \ref{th: main large degree} and \ref{th: main small degree}) we discuss the other possible ranges of degrees. Let us also discuss the properties we assume. Property \ref{p: global expansion} requires the graph $G$ to have very mild `global' edge expansion. Already in \cite{DEKK25} it was shown that if we allow the graph to have too weak `global' edge expansion, then there are examples of $d$-regular graphs such that largest component in the supercritical regime is typically of order $o(n)$.\footnote{Here and throughout the rest of the paper, we say that $a=o(b)$ if $\frac{a}{b}$ tends to zero as $d$ tends to infinity.} Properties \ref{p: ball growth} and \ref{p: small set expansion} ask for some control on the `local' expansion of the graph $G$, and we note that, in fact, it suffices to require these properties for connected sets only. We further remark that we may allow $c_1$ to tend to zero as $C$ tends to infinity. 

Since (as previously mentioned) when $G$ is a $d$-regular graph and $p=\frac{1-\epsilon}{d}$, the largest component of $G_p$ is typically of size $O(\log n/\epsilon^2)$ (see, for example, \cite[Theorem 2]{DEKK22}), we obtain here, essentially, a discrete duality principle for the family of graphs which satisfy the properties of Theorem~\ref{th: main}. Indeed, the family of graphs satisfying the properties of Theorem~\ref{th: main} is quite wide, and includes the $d$-dimensional hypercube, random $d$-regular graphs, and certain families of expanders. We discuss the applicability of Theorem~\ref{th: main} and its two variants, Theorems~\ref{th: main large degree} and \ref{th: main small degree}, as well as relations to previous results in Section~\ref{s: applications}. 

Before discussing the variants of Theorem~\ref{th: main} and applications, let us examine the tightness of its assumptions. Our second main result shows that the `local' assumptions are in fact tight, and thus the characterisation in Theorem~\ref{th: main} is indeed tight.
\begin{theorem}\label{th: construction}
Let $\epsilon>0$ be a sufficiently small constant, and let $c_1\ge 10$ be a constant. Suppose that $d=\omega(1), d=o(n)$, and let $p=\frac{1+\epsilon}{d}$. Then, there are infinitely many pairs $d,n$ for which there exists a $d$-regular graph $G$ on $n$ vertices satisfying Property \ref{p: global expansion}, and that for every $U\subseteq V(G)$ with $|U|\le \frac{\log(n/d)}{40c_1}$, $|N(U)|\ge d|U|$, and that for every $U\subseteq V(G)$ with $|U|\le \frac{\epsilon^3}{100c_1}\cdot d\log(n/d)$, $e(U,U^C)\ge (1-\epsilon^3)d|U|$, and yet \textbf{whp} the second largest component in $G_p$ is of order at least $\frac{\epsilon d\log (n/d)}{30c_1}$.
\end{theorem}
Theorem \ref{th: construction} shows that for $G$ to exhibit the Erd\H{o}s-R\'enyi component phenomenon in full, we indeed need some `local' control on the graph $G$. Moreover, it shows that the assumptions in Property \ref{p: small set expansion} are in fact tight up to a constant factor.  

Let us further mention two variants of Theorem \ref{th: main}, whose proofs are nearly identical to that of Theorem~\ref{th: main}, and which cover the remaining ranges of degrees. The first one considers graphs whose degree is at least $10\log n/\epsilon$ (as opposed to $\log^{\alpha}n$ for some $\alpha>0$), and there we are able to remove assumption \ref{p: ball growth}.
\begin{theorem}\label{th: main large degree}
Let $c_1>0$ be bounded from below and satisfy $c_1=\omega(d/n)$. Let $\epsilon>0$ be a sufficiently small constant. Suppose that $d\ge \frac{10\log n}{\epsilon}$, and let $G$ be a $d$-regular graph on $n$ vertices. Let $p=\frac{1+\epsilon}{d}$. Suppose that $G$ satisfies assumption \ref{p: global expansion} with $c_1$. Then, there exists a constant $C= C(c_1)$ such that if $G$ has that
\begin{enumerate}[(P\arabic*')]
\setcounter{enumi}{1}
    \item for every $U\subseteq V(G)$ with $|U|\le \min\{Cd\log n,\epsilon^5n\}$, $e(U,U^C)\ge (1-\epsilon^3)d|U|$, \label{p': small set expansion}
\end{enumerate}
then \textbf{whp} there exists a unique giant component in $G_p$ whose order is $\left(1+o(1)\right)y(\epsilon)n$, and all the other components in $G_p$ are of order $O\left(\frac{\log n}{\epsilon^2}\right)$.
\end{theorem}
Note that Theorem \ref{th: main large degree} indeed covers every $d\in [10\log n, n-1]$. Let us also note that our requirement that $c_1=\omega(d/n)$ can naturally be omitted when $d=o(n)$, however, it is indeed necessary when $d=\Theta(n)$ --- otherwise, consider a collection of $\frac{n}{d+1}$ disjoint cliques of size $d+1$, weakly connected to each other (that is, with $cd$ edges leaving each clique for some small $c>0$). There, it is fairly `cheap' to disconnect one of the cliques from the rest of the graph after percolation.

Recall that Theorem \ref{th: main} requires that $d\ge \log^{\alpha}n$ for some $\alpha>0$. Our second variation shows that by asking for a strengthening of Property \ref{p: small set expansion}, we can remove both the requirement on the degree and assumption \ref{p: ball growth} altogether.
\begin{theorem}\label{th: main small degree}
Let $\epsilon>0$ be a sufficiently small constant. Suppose that $d=\omega(1)$, and let $G$ be a $d$-regular graph on $n$ vertices. Let $p=\frac{1+\epsilon}{d}$. Suppose that $G$ satisfies Property \ref{p: global expansion} for some $c_1>0$. Suppose furthermore that for every $U\subseteq V(G)$ with $|U|\le (d\log n)^{5\log \log n}$, $e(U,U^C)\ge (1-\epsilon^3)d|U|$. Then, \textbf{whp} there exists a unique giant component in $G_p$ whose order is $\left(1+o(1)\right)y(\epsilon)n$, and all the other components in $G_p$ are of order $O\left(\frac{\log n}{\epsilon^2}\right)$.
\end{theorem}
Note that the assumptions in Theorem \ref{th: main small degree} imply that $(d\log n)^{5\log\log n}=O(n)$ (and thus, say, $d\ll n^{1/(10\log\log n)}$ suffices).

\subsection{Applications and comparison to previous results}\label{s: applications}
Theorem \ref{th: main}, together with its two variants (Theorems \ref{th: main large degree} and \ref{th: main small degree}) recover several classical results, as well as cover new ground. As the proof of Theorem \ref{th: main} and the proofs of Theorems \ref{th: main large degree} and \ref{th: main small degree} are nearly identical, this gives a \textit{unified approach} to establishing the ERCP. The classical result of $G(n,p)$ \cite{ER60}, which can be seen as percolation on the complete graph, is recovered by Theorem \ref{th: main large degree}. More interestingly, the proof here recovers results for percolation on two classical graphs, for which the known proofs quite differ from each other --- pseudo-random $(n,d,\lambda)$-graphs \cite{FKM04}, and the $d$-dimensional hypercube $Q^d$ \cite{AKS81,BKL92}. Indeed, the proof for $(n,d,\lambda)$-graphs \cite{FKM04} relies heavily on their nearly perfect edge distribution, whereas the proof for the hypercube $Q^d$ \cite{AKS81, BKL92} heavily relies on its product structure. 

Frieze, Krivelevich and Martin showed that when $G$ is an $(n,d,\lambda)$-graph with $\lambda=o(d)$ and $p=\frac{1+\epsilon}{d}$, \textbf{whp} $G_p$ exhibits the ERCP. By the classical expander mixing lemma \cite{AC88}, these graphs satisfy that for every $U\subseteq V(G)$ with $|U|\le \epsilon^5n$, $e(U,U^C)\ge (1-\epsilon^3)d|U|$, and thus Theorems \ref{th: main large degree} and \ref{th: main small degree} recover this result for $d\ge 10\log n$ and $d<\log n$, respectively. Since random $d$-regular graphs are known to be \textbf{whp} $(n,d,\Omega(\sqrt{d}))$-graphs (see, for example, \cite{S23}), this shows that random $d$-regular graphs (for growing degree $d$) exhibit the ERCP. In fact, it easily follows from the expander mixing lemma that the vertex expansion of $(n,d,\lambda)$-graphs is by a $\Omega(d^2/\lambda^2)$-factor, and thus \textbf{whp} random $d$-regular graphs have vertex expansion by a $\Omega(d)$-factor, and thus we may apply Theorem \ref{th: main} directly. 

As for the hypercube $Q^d$, the classical isoperimetric results of Harper \cite{H64} show that it satisfies the assumptions of Theorem \ref{th: main}. Indeed, we have there that $n\coloneqq |V(Q^d)|=2^d$, and thus $d=\log_2n$. Furthermore, by \cite{H64}, for every $U\subseteq V(Q^d)$, $e(U,U^C)\ge |U|(d-\log_2|U|)$, and thus assumptions \ref{p: global expansion} and \ref{p: small set expansion} are satisfied in $Q^d$. Finally, by \cite{H64}, for every $U\subseteq V(Q^d)$ with $|U|=d+1$, $|N(U)|\ge \binom{d}{2}$, and thus assumption \ref{p: ball growth} is satisfied in $Q^d$, and Theorem \ref{th: main} shows that $Q^d$ exhibits the ERCP. Moreover, by Theorem \ref{th: main}, other graphs which exhibit Harper-like isoperimetric inequalities exhibit the ERCP. In \cite{DEKK24} it was shown that high-dimensional product graphs have Harper-like isoperimetric inequalities, and thus Theorem \ref{th: main} recovers the fact that they exhibit the ERCP \cite{DEKK22}. It is not hard to see that the duplicube \cite{BDGZ23} also satisfies Harper-like isoperimetric inequalities, and thus Theorem \ref{th: main} shows that it exhibits the ERCP, answering a question of Benjamini and Zhukovskii \cite{BZ23}. Recent work of Collares, Doolittle and Erde \cite{CDE24} showed that the permutahedron exhibits the ERCP, and it would be interesting to see whether this can be expressed in a terms similar to that of Theorem \ref{th: main}.

As mentioned before, when $d$ is constant, Krivelevich, Lubetzky, and Sudakov \cite{KLS20} showed that if $G$ has high-girth, then when $p=\frac{c}{d-1}$ for $c>1$, typically $G_p$ has a giant component of asymptotic order $y(\epsilon)|V(G)|$. They further showed that this requirement alone does not suffice to guarantee that typically the second largest component is of order at most $O_{c}(\log n)$. 

In a companion paper \cite{DK25}, we consider the constant degree case and show that assumptions similar to \ref{p: global expansion} and \ref{p: small set expansion} suffice to establish that typically the second largest component is of order at most $O_{c}(\log n)$.  

\subsection{Structure of the paper}
The structure of the paper is as follows. In Section~\ref{s: prelim} we set out some notation and lemmas which will be of use throughout the paper. Section~\ref{s: main} is devoted to the proof of Theorem~\ref{th: main}. In Section~\ref{s: variants} we explain how to slightly modify the proof of Theorem~\ref{th: main} in order to obtain Theorems~\ref{th: main large degree} and~\ref{th: main small degree}. In Section~\ref{s: construction} we give a construction, proving Theorem~\ref{th: construction}. Finally, in Section~\ref{s: discussion}, we discuss our results and consider avenues for future research.

\section{Preliminaries}\label{s: prelim}
Let us first state some notational conventions used in the paper. Given a graph $H$ and a vertex $v$, we denote by $C_H(v)$ the set of vertices in the connected component of $v$ in $H$. Given $A\subseteq V(H)$, we let $A^C\coloneqq V(H)\setminus A$. We say that a set $A$ is \textit{connected}, if the induced subgraph $H[A]$ is connected. We denote by $d(v,A)$ the number of neighbours of $v$ in $A$ (in $H$), and by $N(A)$ the external neighbourhood of $A$, that is, $N(A)\coloneqq\{u\in A^C\colon \exists v\in A, \{uv\}\in E(H)\}$. Given $v\in V(H)$ and $r\in \mathbb{N}$, we denote by $B_H(v,r)$ the ball of radius $r$ centred in $v$, that is, set of vertices at distance at most $r$ from $v$ in $H$. Given $\Gamma\subseteq H$, and $A,B\subseteq V(\Gamma)$ with $A\cap B=\varnothing$, we denote by $E_{\Gamma}(A,B)$ the set of edges in $\Gamma$ with one endpoint in $A$ and the other endpoint in $B$, and set $e_{\Gamma}(A,B)\coloneqq |E_{\Gamma}(A,B)|$. Further, we denote by $E_{\Gamma}(A)$ all edges in $\Gamma$ with both their endpoints in $A$, and set $e_{\Gamma}(A)\coloneqq|E_{\Gamma}(A)|$. All logarithms are with the natural base unless explicitly stated otherwise. For the sake of clarity, we systemically ignore rounding signs.

We will make use of two standard probabilistic bounds. The first one is a typical Chernoff-type tail bound on the binomial distribution (see, for example, Appendix A in \cite{AS16}).
\begin{lemma}\label{l:  chernoff}
Let $n\in \mathbb{N}$, let $p\in [0,1]$, and let $X\sim Bin(n,p)$. Then for any $0<t\le \frac{np}{2}$, 
\begin{align*}
    &\mathbb{P}\left[|X-np|\ge t\right]\le 2\exp\left\{-\frac{t^2}{3np}\right\}.
\end{align*}
\end{lemma}

The second one is a variant of the well-known Azuma-Hoeffding inequality (see, for example, Chapter 7 in \cite{AS16}),
\begin{lemma}\label{l: azuma}
Let $m\in \mathbb{N}$ and let $p\in [0,1]$. Let $X = (X_1,X_2,\ldots, X_m)$ be a random vector with range $\Lambda = \{0,1\}^m$ with each $X_{\ell}$ distributed according to independent Bernoulli$(p)$. Let $f:\Lambda\to\mathbb{R}$ be such that there exists $C \in \mathbb{R}$ such that for every $x,x' \in \Lambda$ which differ only in one coordinate,
\begin{align*}
    |f(x)-f(x')|\le C.
\end{align*}
Then, for every $t\ge 0$,
\begin{align*}
    \mathbb{P}\left[\big|f(X)-\mathbb{E}\left[f(X)\right]\big|\ge t\right]\le 2\exp\left\{-\frac{t^2}{2C^2mp}\right\}.
\end{align*}
\end{lemma}

We also require the following bound on the number of $k$-vertex trees in a graph
$G$, which follows immediately from \cite[Lemma 2]{BFM98}.
\begin{lemma}\label{l: trees}
Let $G$ be a graph of maximum degree at most $d$, let $v\in V(G)$, and let $k \in \mathbb{N}$. Let $t_{k}(G,v)$ be the number of trees on $k$ vertices in $G$ rooted at $v$. Then
\begin{align*}
     t_k(G,v)\le \frac{k^{k-2}d^{k-1}}{k!}\le (ed)^{k-1}.
\end{align*}
\end{lemma}

We will utilise the following lemma, allowing one to find large matchings in percolated subgraphs; it follows immediately from \cite[Lemma 3.8]{DEKK24}.
\begin{lemma}\label{l: matching}
Let $G$ be a $d$-regular graph. Let $0<\delta<\frac{1}{2}$ be a constant, and let $s\ge \Omega(d)$. Let $F\subseteq E(G)$ be such that $|F|= s$, and let $q=\frac{\delta}{d}$. Then, $F_{q}$, a random subset of $F$ obtained by retaining each edge independently with probability $q$, contains a matching of size at least $\frac{\delta^2s}{d}$ with probability at least $1-\exp\left\{-\frac{\delta^2s}{d}\right\}$.
\end{lemma}

\section{Proof of Theorem \ref{th: main}}\label{s: main}
Throughout this section, we assume that $G=(V,E)$ satisfies the assumptions of Theorem \ref{th: main}. Given $H\subseteq G$, let $V_L(H)$ be the set of vertices in large components in $H$, that is,
\begin{align}
    V_L(H)\coloneqq \left\{v\in V\colon |C_{H}(v)|\ge \frac{7\log n}{\epsilon^2}\right\}. \label{definition of large}
\end{align}
Let us first briefly discuss the proof's strategy. We will utilise a double-exposure/sprinkling argument. We set $p_2=\frac{\epsilon^3}{d}$, and let $p_1$ be such that $(1-p_1)(1-p_2)=1-p$. Note that $p_1\ge \frac{1+\epsilon-\epsilon^3}{d}$, and that $G_p$ has the same distribution as $G_{p_1}\cup G_{p_2}$. (We remark here that the choice of value for $p_2$ is tightly related to the Taylor series of $\log(1+x)$ at $x=0$, as will become clear in the proof of Lemma~\ref{l: gap}).
In Section \ref{s: dense}, we show that large components in $G_{p_1}$ are typically 'everywhere dense' in the sense that \textbf{whp} for every $v\in V$, $$|B_G(v,\log_d\log n+1)\cap V_L(G_{p_1})|=\Omega(d\log n),$$
that is, we show that \textbf{whp} for every vertex $v\in V$, a constant proportion of the vertices in a ball of radius $\log_d\log n+1$ centred in $v$ are in large components in $G_{p_1}$. We note that already Ajtai, Koml\'os, and Szemeredi~\cite{AKS81} used this broad strategy -- showing that `large components' are `everywhere dense' -- to establish the phase transition in the hypercube. Here, however, our graph does not necessarily have any product structure, and our assumptions are only on the expansion properties of the graph. 

We continue by showing that typically there are no components in $G_{p_1}$ whose order is between $\frac{7\log n}{\epsilon^2}$ and $Cd\log n$. Then, we show in Lemma \ref{l: merge} that typically all components of order at least $\Omega(d\log n)$ in $G_{p_1}$ merge after sprinkling with probability $p_2$. This requires some delicate treatment, as we might need to find paths of length $\Omega(\log_d\log n)$ in $G_{p_2}$, and the probability a path of length $\ell$ is in $G_{p_2}$ is $p_2^{\ell}$. We further show that \textbf{whp} all components in $G_{p_1}\cup G_{p_2}$, besides those intersecting with $V_L(G_{p_1})$, are of order at most $\frac{14\log n}{\epsilon^2}$ (we note that the difference in the constant between here and the definition of large components above is intentional). This also requires some careful treatment --- see Lemma \ref{l: gap} and the proof of Theorem~\ref{th: main}. Finally, we show that the total volume of vertices in components of order at least $\frac{14\log n}{\epsilon^2}$ in $G_p$ is of asymptotic order $y(\epsilon)n$.  

Recall that $c_3\in (0,1]$, and note that we may assume throughout the proof, without loss of generality, that $c_1,c_2\le 1$.

\subsection{Large components are `everywhere dense'}\label{s: dense}
Throughout this section, we state the results for $G_p$, however they all follow through for $G_{p_1}$ as well, with the natural substitution of $\epsilon\to\epsilon-\epsilon^3$.

Let $c'\coloneqq c_2c_3^{1+\frac{1}{\alpha}}$. This choice of a constant may seem peculiar at first, but we will shortly see its importance. We begin by analysing the typical behaviour of components incident to a fixed set of order $c'd\log n$. 
\begin{lemma}\label{l: components on logarithmic sets}
Let $S\subseteq V$ with $|S|=c'd\log n$. Then, the following holds. 
\begin{enumerate}[(a)]
    \item The probability there exists $U\subseteq S$ such that
        \begin{align*}
            \left|\bigcup_{u\in U}C_{G_p}(u)\right|\in \left[\frac{c'd\log n}{\epsilon^3}, \frac{2c'd\log n}{\epsilon^3}\right],
        \end{align*}
        is at most $o\left(1/n\right)$. \label{item: comp of odd volume}
    \item The probability there exists $U\subseteq S$ with $|U|\ge (1-\epsilon^2)c'd\log n$ such that
    \begin{align*}
        \left|\bigcup_{u\in U}C_{G_p}(u)\right|\le Cd\log n,
    \end{align*}
    is at most $o\left(1/n\right)$. \label{item: everything grows}
\end{enumerate}
\end{lemma}
\begin{proof} Let $s\coloneqq |S|$.
\begin{enumerate}[(a)]
    \item We restrict ourselves to $U'\subseteq U$, such that $C_{G_p}(u)$ are disjoint for each $u\in U'$. Let $F$ be a spanning forest of the components meeting $U'$ in $G_p$, such that $|V(F)|=k\in\left[\frac{s}{\epsilon^3}, \frac{2s}{\epsilon^3}\right]$. The forest is composed of some $\ell\coloneqq |U'|\le s$ tree components, $T_1,\ldots, T_{\ell}$, where for every $j_1\neq j_2$, $V(T_{j_1})\cap V(T_{j_2})=\varnothing$.
    We have that all the edges leaving $F$ are not in $G_p$, and that each $T_j$ contains a unique vertex $u_j\in U'$ for $1\le j \le \ell$. Note that if there is a subset $U\subseteq S$ satisfying the conditions of the lemma, then such an $F$ exists. Indeed, one can decompose $\bigcup_{u\in U}C_{G_p}(u)$ into disjoint connected sets, and for each such set, choose one vertex in $U$, thus forming $U'$ and the forest $F$.

    Let us now bound from above the probability such a forest $F$ exists. We specify $\ell\in [s]$, the size $U'$. Then, there are $\binom{s}{\ell}\le 2^s$ ways to choose $U'\subseteq S$. We can then specify the forest $F$ by choosing $|V(F)|=k\in \left[\frac{s}{\epsilon^3}, \frac{2s}{\epsilon^3}\right]$, the sizes of the tree components $|T_i|=k_i$ such that $\sum_{i=1}^{\ell}k_i=k$, and finally the tree components $\{B_1,\ldots, B_{\ell}\}$, for which by Lemma \ref{l: trees} there are at most $\prod_{i=1}^{\ell}(ed)^{k_i-1}=(ed)^{k-\ell}$ choices altogether. For a fixed forest $F$ with $\ell$ components there are $k-\ell$ edges which must appear in $G_p$, which happens with probability $p^{k-\ell}$. Since $k\le \frac{2s}{\epsilon^3}<Cd\log n$ (for $C$ large enough with respect to $c_2, c_3, \alpha, \epsilon$), by \ref{p: small set expansion} there are at least $(1-\epsilon^3)kd$ edges in the boundary of $V(F)$ which must not appear in $G_p$, which happens with probability at most $(1-p)^{(1-\epsilon^3)kd}$. Thus, by the union bound, the probability such $F$ exists is at most
    \begin{align*}
        \sum_{\ell=1}^{s}2^{s}\sum_{k=s/\epsilon^3}^{2s/\epsilon^3}\sum_{\substack{k_1,\ldots,k_{\ell}>0\\ k_1+\cdots+k_{\ell}=k}}(ed)^{k-\ell}p^{k-\ell}(1-p)^{(1-\epsilon^3)kd}.
    \end{align*}
    We have that
    \begin{align*}
        (ed)^{k-\ell}p^{k-\ell}(1-p)^{(1-\epsilon^3)kd}&\le \left[e(1+\epsilon)\exp\left\{-(1+\epsilon)(1-\epsilon^3)\right\}\right]^k\\
        &\le \exp\left\{-\frac{\epsilon^2k}{4}\right\}\\
        &\le \exp\left\{-\frac{\epsilon^2s}{4\epsilon^3}\right\}\le \exp\left\{-\frac{s}{4\epsilon}\right\},
    \end{align*}
    where we used $1+x\le \exp\left\{x-\frac{x^2}{3}\right\}$ for small enough $x$, and that $k\ge \frac{s}{\epsilon^3}$. There are $\binom{k+\ell-1}{\ell}$ ways to choose $k_1,\ldots, k_{\ell}>0$ such that $\sum_{i=1}^{\ell}k_i=k$. Recalling that $k\in \left[\frac{s}{\epsilon^3}, \frac{2s}{\epsilon^3}\right]$ and $\ell\in[s]$, we have that $\binom{k+\ell-1}{\ell}\le \left(\frac{e\cdot3s}{s\cdot\epsilon^3}\right)^{s}\le \left(\frac{9}{\epsilon^3}\right)^{s}$. Altogether, the probability that such $F$ exists is at most
    \begin{align*}
        s2^{s}\cdot \frac{2s}{\epsilon^3}\cdot \left(\frac{9}{\epsilon^3}\right)^{s}\exp\left\{-\frac{s}{4\epsilon}\right\}&\le \exp\left\{s\left(1+\log(9/\epsilon^3)-\frac{1}{4\epsilon}\right)\right\}\\&=o\left(1/n\right),
    \end{align*}
    where in the last equality we used that $s=c'd\log n$ and that $\epsilon$ is sufficiently small.
    \item Fix $U\subseteq S$, with $k=|U|\ge (1-\epsilon^2)s$. We will utilise a variant of the Breadth First Search (BFS) exploration process. To that end, let $(X_i)_{i=1}^{nd/2}$ be a sequence of i.i.d. Bernoulli$(p)$ random variables, and assume some order $\sigma$ on the vertices of $G$. We maintain three sets of vertices: $W$, the set of vertices whose exploration has been completed; $A$, the set of active vertices, kept as a queue; and, $Y$, the set of vertices which we have yet to explore. We initialise $W=\varnothing$, $A=U$, and $Y=V\setminus U$, and let $v_A$ be the first vertex in $A$. The algorithm stops once $A$ is empty. If at moment $t$ (that is, when we query the $t$-th edge) the set $A$ is still non-empty, we consider the first (according to $\sigma$) neighbour (in $G$) of $v_A$ in $Y$. If $X_t=1$, we move this vertex from $Y$ to $A$, and continue. If $X_t=0$, we consider the next (according $\sigma$) neighbour (in $G$) of $v_A$ in $Y$. If there are no remaining neighbours of $v_A$ in $Y$, we move $v_A$ from $A$ to $W$, and continue. Note that, as in the BFS algorithm, we received a negative answer for all the queries about the edges between $W$ and $Y$ at all times. 
    Furthermore, once $A$ is empty, we have that $G_p[W]$ has the same distribution as $\cup_{u\in U}C_{G_p}(u)$.    

    Suppose that $A$ is empty at some moment $t$ and that at that moment $|W|=w\le Cd\log n$. Then, by \ref{p: small set expansion}, we have that $t\ge e_G(W,Y)=e_G(W,W^C)\ge w(1-\epsilon^3)d$, and we have received $w-k$ positive answers. Thus, by Lemma \ref{l: chernoff} the probability of this event is at most
    \begin{align*}
        \mathbb{P}\left[Bin\left(w(1-\epsilon^3)d,\frac{1+\epsilon}{d}\right)\le w-k\right]\le \exp\left\{-\frac{(\epsilon w/2+k)^2}{4w}\right\}\le \exp\left\{-\epsilon k\right\}.
    \end{align*}
    We have at most $n$ ways to choose $w$, and at most 
    \begin{align*}
        \sum_{\ell=(1-\epsilon^2)s}^{s}\binom{s}{\ell}=\sum_{\ell=0}^{\epsilon^2 s}\binom{s}{\ell}\le \left(\frac{e}{\epsilon^2}\right)^{\epsilon^2 s}
    \end{align*}
    ways to choose $U\subseteq S$. Thus, by the union bound, the probability of this event is at most
    \begin{align*}
        n\cdot \left(\frac{e}{\epsilon^2}\right)^{\epsilon^2 s}\exp\left\{-\epsilon(1-2\epsilon)s\right\}&= \exp\left\{2\log n+\epsilon^2\log(e/\epsilon^2)s-\epsilon(1-2\epsilon)s/4\right\}=o\left(1/n\right).
    \end{align*}
\end{enumerate}
\end{proof}
The following is an almost immediate conclusion of Lemma \ref{l: components on logarithmic sets} (recall here the definition of $V_L(H)$, given in \eqref{definition of large}).
\begin{corollary}\label{col: components on logarithmic sets}
Let $S\subseteq V$ with $|S|=c'd\log n$. Then, with probability at least $1-o\left(1/n\right)$ there exists a subset $X\subseteq S$, with $|X|\ge \epsilon^2 c'd\log n$, such that $X\subseteq V_L(G_p)$.
\end{corollary}
\begin{proof}
By Lemma \ref{l: components on logarithmic sets}\ref{item: everything grows}, with probability at least $1-o\left(1/n\right)$, every subset $U\subseteq S$ with $|U|\ge (1-\epsilon^2)c'd\log n$ satisfies 
\begin{align*}
    \left|\bigcup_{u\in U}C_{G_p}(u)\right|\ge Cd\log n. 
\end{align*}
Furthermore, we claim that by Lemma \ref{l: components on logarithmic sets}\ref{item: comp of odd volume}, with probability at least $1-o\left(1/n\right)$, we have that
\begin{align*}
    \left|\left(V\setminus V_L(G_p)\right)\cap \bigcup_{u\in U}C_{G_p}(u)\right|\le \frac{c'd\log n}{\epsilon^3}.
\end{align*}
Indeed, suppose that $\left|\left(V\setminus V_L(G_p)\right)\cap \bigcup_{u\in U}C_{G_p}(u)\right|\ge 1+c'd\log n/\epsilon^3$. Then, we may associate these components with roots $u_1,\ldots, u_{m}\in U$ for some $m$. For every $i\in [m]$, by definition of $V_L(G_p)$ we have that $|C_{G_p}(u_i)|\le \frac{7\log n}{\epsilon^2}$. Thus, sequentially adding the volumes of these components, at the first moment we crossed the total volume of $\frac{c'd\log n}{\epsilon^3}$, we have a total volume of at most $\frac{c'd\log n}{\epsilon^3}+\frac{7\log n}{\epsilon^2}$, contradicting the assertion of Lemma \ref{l: components on logarithmic sets}\ref{item: comp of odd volume}.

Hence, assuming that $C>\frac{c'}{\epsilon^3}$, with probability at least $1-o\left(1/n\right)$, every subset $U\subseteq S$ of size at least $(1-\epsilon^2)c'd\log n$ has at least one vertex in $V_L(G_p)$. Thus, with probability $1-o(1/n)$, there are less than $(1-\epsilon^2)c'd\log n$ vertices in $S$ which are in $V\setminus V_L(G_p)$, and therefore there are at least $\epsilon^2 c'd\log n$ vertices in $S$ which are in $V_L(G_p)$.
\end{proof}

We are now ready to show that large components are typically `everywhere dense'. 
\begin{lemma}\label{l: every vertex is close to a large component}
\textbf{Whp} every $v\in V$ is at distance (in $G$) at most $1+\log_d\log n$ from at least $\epsilon^2 c'd\log n$ vertices in large components in $G_p$.
\end{lemma}
\begin{proof}
Fix $v\in V$. By \ref{p: ball growth}, 
\begin{align*}
    |B_G\left(v,\log_d\log n\right)|\ge \min\left\{c_2\log n,(c_3d)^{\log_d\log n}\right\}\ge c_2c_3^{\log_d\log n}\log n\ge c_2c_3^{\frac{1}{\alpha}}\log n,   
\end{align*}
where we used that $c_2, c_3\in (0,1]$ and that $d\ge \log^{\alpha}n$. Furthermore, as we can always find a subset of size $c_2c_3^{\frac{1}{\alpha}}\log n\le c_2\log n$ in $B_G\left(v,\log_d\log n\right)$, by \ref{p: ball growth},
\begin{align*}
    |B_G\left(v,1+\log_d\log n\right)|\ge c_3d\cdot c_2c_3^{\frac{1}{\alpha}}\log n=c'd\log n.
\end{align*}
Hence, we can let $S_v\subseteq B_G\left(v,1+\log_d\log n\right)$ be an arbitrary set of order $c'd\log n$. By Corollary \ref{col: components on logarithmic sets}, the probability that $|S_v\cap V_L(G_p)|\le \epsilon^2 c'd\log n$ is at most $o(1/n)$. Union bound over the $n$ choices of $v$ completes the proof.
\end{proof}

\subsection{Large components typically merge}
We continue with double-exposure. Recall that $p_2=\frac{\epsilon^3}{d}$, $p_1\ge \frac{1+\epsilon-\epsilon^3}{d}$, and that $G_{p_1}\cup G_{p_2}$ has the same distribution as $G_p$. Further, recall that $c'=c_2c_3^{1+\frac{1}{\alpha}}$, and let $r\coloneqq 1+\log_{d}\log n$, noting that $r\le 1+\frac{1}{\alpha}$ by our assumption on $d$. 

Let us first show a `gap' in the component sizes.
\begin{lemma}\label{l: gap} 
\textbf{Whp}, there is no connected set $K$ in $G_p$ with $|V(K)|\in \left[\frac{7\log n}{\epsilon^2}, Cd\log n\right]$ and $E_{G_{p_1}}(K,K^C)=\varnothing$. 
\end{lemma}
Note that the above implies that typically there are no components whose order lies in the interval $\left[\frac{7\log n}{\epsilon^2},Cd\log n\right]$, both in $G_{p_1}$ and in $G_p$.
\begin{proof}
Let $k\in \left[\frac{7\log n}{\epsilon^2}, Cd\log n\right]$. Let $\mathcal{A}_{k}$ be the event that there is a connected set $K$ of order $k$ in $G_p$, with $E_{G_{p_1}}(K,K^C)=\varnothing$. Since every connected set has a spanning tree, if $\mathcal{A}_k$ occurs, then there exists some tree $T$ of order $k$ in $G$, such that all its edges are in $G_p$, and none of the edges in $E_G\left(V(T),V\setminus V(T)\right)$ are in $G_{p_1}$. By our assumption, $k\le Cd\log n$ and thus by \ref{p: small set expansion}, $e_G\left(V(T), V\setminus V(T)\right)\ge(1-\epsilon^3)kd$. By Lemma \ref{l: trees}, there are at most $n(ed)^{k-1}$ trees of order $k$ at $G$. Therefore, by the union bound,
\begin{align*}
    \mathbb{P}\left[\mathcal{A}_{k}\right]&\le \sum_{T \text{ is a tree in }G, |V(T)|=k}p^{k-1}(1-p_1)^{e\left(V(T), V\setminus V(T)\right)}\\
                                        &\le \sum_{T \text{ is a tree in }G, |V(T)|=k}p^{k-1}(1-p_1)^{(1-\epsilon^3)kd}\\
                                        &\le n(edp)^{k-1}\exp\left\{-p_1(1-\epsilon^3)dk\right\}.
\end{align*}
Since $p_1\ge \frac{1+\epsilon-\epsilon^3}{d}$, we obtain:
\begin{align*}
    \mathbb{P}\left[\mathcal{A}_k\right]&\le n(e(1+\epsilon))^{k-1}\exp\left\{-(1+\epsilon-\epsilon^3)(1-\epsilon^3)k\right\}\\
                                        &\le n\exp\left\{k\left(1+\log(1+\epsilon)-(1+\epsilon-3\epsilon^3)\right)\right\}\\
                                        &\le n\exp\left\{k\left(1+(\epsilon-\frac{\epsilon^2}{2}+\epsilon^3)-(1+\epsilon-3\epsilon^3)\right)\right\}\\
                                        &\le n\exp\left\{-\frac{\epsilon^2k}{3}\right\}=o(1/n),
\end{align*}
where the last inequality holds for small enough $\epsilon>0$, and the equality follows since $k\ge \frac{7\log n}{\epsilon^2}$. Union bound over the less than $Cd\log n<n$ relevant values of $k$ completes the proof.
\end{proof}

The following lemma shows that all large components in $G_{p_1}$ typically merge after sprinkling with $p_2$.
\begin{lemma}\label{l: merge}
\textbf{Whp}, all the components in $G_{p_1}[V_L(G_{p_1})]$ belong to the same component in $G_{p_1}\cup G_{p_2}$.
\end{lemma}
\begin{proof}
Let $W=V_L(G_{p_1})$. It suffices to show that \textbf{whp}, for every component-respecting partition of $W=A\sqcup B$, there exist a path in $G_{P_2}$ between $A$ and $B$. We may assume that $|A|\le |B|$. By Lemma \ref{l: every vertex is close to a large component} \textbf{whp} every $v\in V$ is at distance at most $r$ from at least $(\epsilon-\epsilon^3)^2c'd\log n\ge \frac{\epsilon^2 c'd\log n}{2}$ vertices in $W$. We continue assuming this holds deterministically.

Let $A_0\coloneqq A$ and let $B_0\coloneqq B$. We define sets $A_1,\ldots, A_{r}$ and $B_1,\ldots, B_{r}$ inductively in the following manner. 
\begin{align*}
    &A_i\coloneqq \left\{v\in V\setminus\bigcup_{0\le j <i}(A_j\cup B_j)\colon d(v, A_{i-1})\ge \frac{\epsilon^2 c'd}{5r}\right\},\\
    &B_i\coloneqq \left\{v\in V\setminus\left(\bigcup_{0\le j <i}(A_j\cup B_j)\cup A_i\right)\colon d(v, B_{i-1})\ge \frac{\epsilon^2 c'd}{5r}\right\}.
\end{align*}
Let $A'=\bigcup_{i=0}^{r}A_i$ and $B'=\bigcup_{i=0}^rB_i$. We will make use of the following claim, whose proof we postpone to the end of this proof.
\begin{claim}\label{c: split the world}
    Suppose that every $v\in V$ is at distance at most $r$ from at least $\frac{\epsilon^2 c'd\log n}{2}$ vertices in $W$. Then $V=A'\sqcup B'$.
\end{claim}

We first expose the edges between $A'$ and $B'=V\setminus A'$ in $G_{p_2}$. By Property \ref{p: global expansion}, $e(A', B')\ge c_1|A|$. Thus, by Lemma \ref{l: matching} (applied with $\delta_{\ref{l: matching}}=\epsilon^3$), with probability at least $1-\exp\left\{-\epsilon^6c_1|A|/d\right\}$, there exists a matching $M$ of order at least $\epsilon^6c_1|A|/d$ between $A'$ and $B'$ in $G_{p_2}$. Let $M_{A'}$ be the endpoints of this matching in $A'$. By the pigeonhole principle, we have that there is some $i\in [r]$ for which $|M_{A'}\cap A_i|\ge \frac{1}{r}|M_{A'}|$. We may assume that $i=r$, as the other cases follow more easily, and let $M_{r,A'}\coloneqq M_{A'}\cap A_r$. 

We now expose the edges between $A_r$ and $A_{r-1}$ in $G_{p_2}$. By construction, there are at least $\frac{\epsilon^2 c'd}{5r}\cdot |M_{r,A'}|$ edges between $M_{r,A'}$ and $A_{r-1}$, and thus by Lemma \ref{l: matching}, with probability at least $1-\exp\left\{-\frac{\epsilon^6}{d}\cdot \frac{\epsilon^2 c'd}{5r} |M_{r,A'}|\right\}=1-\exp\left\{-\frac{\epsilon^8c'|M_{r,A'}|}{5r}\right\}$, there exists a matching of size at least $\frac{\epsilon^8c'|M_{r,A'}|}{5r}$ between $M_{r,A'}$ and $A_{r-1}$. We denote its endpoints in $A_{r-1}$ by $M_{r-1,A'}$. In this manner, we have extended the matching $M$ between $A'$ and $B'$, to vertex-disjoint paths in $G_{p_2}$ of length two between $B'$ and $A_{r-1}$.

We now proceed inductively. Suppose we found $|M_{i,A'}|$ vertex-disjoint paths in $G_{p_2}$ of length $r-i$ between $B'$ and $A_{i}$, where we denote by $M_{i,A'}$ the endpoints of these vertex-disjoint paths in $A_i$. We now expose the edges between $A_{i}$ and $A_{i-1}$ in $G_{p_2}$. Similarly to before, by Lemma \ref{l: matching} and by the definition of $A_i$, with probability at least $1-\exp\left\{-\frac{\epsilon^8c'|M_{i,A'}|}{5r}\right\}$ there is a matching of size at least $\frac{\epsilon^8c'|M_{i,A'}|}{5r}$ between $M_{i,A'}$ to $A_{i-1}$ in $G_{p_2}$, and we denote the set of its endpoints in $A_{i-1}$ by $M_{{i-1},A'}$. We thus extend the set of vertex-disjoint paths, with all their edges in $G_{p_2}$, into $A_{i-1}$. Once reaching to $A_0$, we have obtained that with probability at least
\begin{align*}
    1-\sum_{i=0}^r\exp\left\{-\frac{\epsilon^8c'|M_{i,A'}|}{5r}\right\}-\exp\left\{-\frac{\epsilon^6c_1|A|}{d}\right\},
\end{align*}
there are at least
\begin{align*}
    \frac{\epsilon^7c'|M_{1,A'}|}{5r}\ge \frac{(\epsilon^8c')^{r}|M_{r,A'}|}{(5r)^{r}}\ge \left(\frac{\epsilon^8c'}{5r}\right)^{r}\cdot\frac{\epsilon^6c_1|A|}{rd}\ge \left(\frac{\epsilon^8c'c_1}{5r}\right)^{r+1}\frac{|A|}{d}
\end{align*}
vertex disjoint paths in $G_{p_2}$ of length $r+1$ between $B'$ and $A_0=A$. Let $M_{B'}$ be the corresponding endpoints in $B'$ of these paths. We now repeat the same argument inside $B'$, and conclude that with probability at least
\begin{align*}
    1-2\exp\left\{-\left(\frac{\epsilon^8c'c_1}{5r}\right)^{2r+1}\frac{|A|}{d}\right\},
\end{align*}
there is a path in $G_{p_2}$ between $A$ and $B$.

By Lemma \ref{l: gap}, \textbf{whp} every component in $W$ is of order at least $Cd\log n$. Thus, there are at most $\binom{n/Cd\log n}{|A|/Cd\log n}\le n^{|A|/Cd\log n}$ ways to partition $W$ into $A\sqcup B$. Hence, by the union bound, the probability there is such a partition without a path in $G_{p_2}$ between $A$ and $B$ is at most
\begin{align*}
    \sum_{|A|=Cd\log n}^{n/2}\binom{n/Cd\log n}{|A|/Cd\log n}&2\exp\left\{-\left(\frac{\epsilon^8c'c_1}{5r}\right)^{2r+1}\frac{|A|}{d}\right\}\\&\le 2\sum_{|A|=Cd\log n}^{n/2}\exp\left\{\frac{|A|}{d}\left(\frac{1}{C}-\left(\frac{\epsilon^8c'c_1}{5r}\right)^{2r+1}\right)\right\}\\
    &\le 2n\cdot\exp\left\{-\left(\frac{\epsilon^8c'c_1}{7r}\right)^{2r+1}C\log n\right\}=o(1),
\end{align*}
where we used that $|A|\ge Cd\log n$, and that $C$ is large enough with respect to $\epsilon, c_1, c_2, c_3$ and $\alpha$, and recalling that $r\le 1+\frac{1}{\alpha}$.
\end{proof}
\begin{proof}[Proof of Claim \ref{c: split the world}]
By definition, we have that $A'\cap B'=\varnothing$. For every $i\in [1,r]$, we claim that if $v\notin A_i\cup B_i$, then there are at most $\frac{2i\epsilon c'}{5r}d^i$ vertices in $W$ at distance exactly $i$ from $v$. Note that this implies that $V=A'\cup B'$ --- indeed, if $v\notin A'\cup B'$, then, in particular, $v\notin A_{r}\cup B_{r}$, and thus $G$ has at most $\frac{2\epsilon c'}{5}d^r=\frac{2}{5}\epsilon^2 c'd\log n$ vertices in $W$ at distance exactly $r$ from $v$, and at most $O(d^{r-1})=O(\log n)$ vertices in $W$ at distance at most $r-1$ from $v$. Since $\frac{2}{5}\epsilon^2 c'd\log n+O(\log n)<\frac{\epsilon^2 c' d\log n}{2}$, this contradicts our assumption.

We proceed by induction on $i$. For $i=1$, since $v\notin A_1\cup B_1$, it has at most $2\cdot\frac{\epsilon^2 c'}{5r}d$ neighbours in $W$. Assume the claim is true for $i\in [1,r-1]$, and let us show that it holds for $i+1$. Since $v\notin A_{i+1}\cup B_{i+1}$, it has at most $2\frac{\epsilon c'd}{5r}$ neighbours in $A_i\cup B_i$, which contribute at most $\frac{2\epsilon^2 c'd}{5r}d^{i}$ vertices in $W$ at distance $i+1$ from $v$. Moreover, $v$ has at most $d$ neighbours not in $A_i\cup B_i$, which by the induction hypothesis contribute at most $d\cdot \frac{2i \epsilon^2 c'}{5r}d^i$ vertices in $W$ at distance $i+1$ from $v$. Thus, altogether, $v$ has at most
\begin{align*}
    \frac{2\epsilon^2 c'd}{5r}d^{i}+d\cdot \frac{2i \epsilon c'}{5r}d^i=\frac{2(i+1)c'}{5r}d^{i+1}
\end{align*}
vertices in $W$ at distance $i+1$ from $v$, as required.
\end{proof}

\subsection{Concentration of the number of vertices in large components}
We require the following bound on the probability that a fixed vertex belongs to a component of `medium' order in $G_p$.
\begin{lemma}\label{l: small gap}
Fix $v\in V$. \textbf{Whp}, $v$ does not belong to a component whose order is between $\sqrt{d}$ and $\frac{14\log n}{\epsilon^2}$. 
\end{lemma}
\begin{proof}
Let $k\coloneqq |C_{G_p}(v)|$, and suppose that $k\in \left[\sqrt{d}, \frac{14\log n}{\epsilon^2}\right]$. Since every component has a spanning tree, there exists some tree $T$ of order $k$ rooted at $v$ in $G$, such that all its edges are in $G_p$, and none of the edges in $E\left(V(T)T,V\setminus V(T)\right)$ are in $G_{p}$. Since $|V(T)|=k\le\frac{14\log n}{\epsilon^2}\le Cd\log n$, we have that $E\left(V(T),V\setminus V(T)\right)\ge (1-\epsilon^3)dk$. Thus, by Lemma \ref{l: trees} and by the union bound, the probability of this event is at most
\begin{align*}
    (ed)^{k-1}p^{k-1}(1-p)^{(1-\epsilon^3)dk}&\le \left[e(1+\epsilon)\exp\left\{-(1+\epsilon)(1-\epsilon^3)\right\}\right]^k\\
     &\le \exp\left\{k\left(1+\epsilon-\frac{\epsilon^2}{3}-(1+\epsilon-2\epsilon^3)\right)\right\}\\
        &\le \exp\left\{-\frac{\epsilon^2k}{4}\right\}=o(1),
\end{align*}
where we used $1+x\le \exp\left\{x-\frac{x^2}{3}\right\}$ for small enough $x$, and that $k\ge\sqrt{d}$.
\end{proof}
We note that in Lemma \ref{l: small gap}, one can in fact show that, fixing $v\in V$, \textbf{whp} $v$ does not belong to a component whose order is between $\sqrt{d}$ and $Cd\log n$, but we only require the statement of Lemma \ref{l: small gap} to proceed. Indeed, we are now ready to show that the set of vertices belonging to large components in $G_p$ is of the correct asymptotic order.
\begin{lemma}\label{l: concentration}
Let $W$ be the set of vertices belonging to components in $G_p$ whose order is at least $\frac{14\log n}{\epsilon^2}$. Then, \textbf{whp}, $|W|=(1+o(1))y(\epsilon)n$.
\end{lemma}
\begin{proof}
Let us first show that $\mathbb{E}\left[W\right]=(1+o(1))y(\epsilon)n$. To that end, fix $v\in V$ and let us estimate $\mathbb{P}\left[|C_{G_p}(v)|\ge \frac{14\log n}{\epsilon^2}\right]$. Run the BFS algorithm rooted at $v$. Since $G$ is $d$-regular, this BFS exploration is stochastically dominated by a Galton-Watson tree with offspring distribution $Bin(d,p)$. Since $dp=1+\epsilon$, standard results (see, for example, \cite[Theorem 4.3.12]{D19}) imply that $\mathbb{P}\left[|C_{G_p}(v)|\ge \frac{14\log n}{\epsilon^2}\right]\le (1+o(1))y(\epsilon)$. On the other hand, consider the BFS exploration with the following alteration --- we terminate the process either once $|C_{G_p}(v)|$ is uncovered, or once we have discovered $\sqrt{d}$ vertices. Then, during the exploration process every vertex in the queue has at least $d-\sqrt{d}$ neighbours in $G$, and thus this BFS exploration stochastically dominates a Galton-Watson tree with offspring distribution $Bin(d-\sqrt{d},p)$. Since $(d-\sqrt{d})p=1+\epsilon-o(1)$, we have by standard results that $\mathbb{P}\left[|C_{G_p}(v)|\ge \sqrt{d}\right]\ge (1-o(1))y(\epsilon)$. Thus, by Lemma \ref{l: small gap}, $\mathbb{P}\left[|C_{G_p}(v)|\ge \frac{14\log n}{\epsilon^2}\right]\ge (1-o(1))y(\epsilon)$. Thus $\mathbb{E}\left[W\right]=(1+o(1))y(\epsilon)n$.

To show that $|W|$ is well concentrated around its mean, consider the standard edge-exposure martingale. Every edge can change the value of $|W|$ by at most $\frac{28\log n}{\epsilon^2}$. Thus, by Lemma \ref{l: azuma},
\begin{align*}
    \mathbb{P}\left[\left||W|-\mathbb{E}[|W|]\right|\ge n^{2/3}\right]\le 2\exp\left\{-\frac{n^{4/3}}{2\cdot\frac{ndp}{2}\cdot\frac{(28\log n)^2}{\epsilon^4}}\right\}=o(1).
\end{align*}
Therefore, \textbf{whp} $|W|=(1+o(1))y(\epsilon)n$, as required.
\end{proof}
We refer the reader to \cite{H23}, where the relation between the above calculation and bounds on branching processes is studied (and, in particular, a universal bound on the size of the largest component of a sequence of graphs converging locally to some finite graph is given).

\subsection{Proof of Theorem \ref{th: main}}
Theorem \ref{th: main} will now follow from Lemmas \ref{l: gap}, \ref{l: merge} and \ref{l: concentration}.

By Lemma \ref{l: gap}, \textbf{whp} there are no components in $G_{p_1}$ whose order is between $\frac{7\log n}{\epsilon^2}$ and $Cd\log n$. By Lemma \ref{l: merge}, \textbf{whp} all the components whose order was at least $\frac{7\log n}{\epsilon^2}$ in $G_{p_1}$ merge into a unique component in $G_p$. Hence, if there exists in $G_p$ outside $V_L(G_{p_1})$ a component whose order is at least $\frac{14\log n}{\epsilon^2}$, it contains components in $G_{p_1}$ whose order (in $G_{p_1}$) is at most $\frac{7\log n}{\epsilon^2}$, and we can then find a set $K$ whose order is between $\frac{7\log n}{\epsilon^2}$ and $\frac{14\log n}{\epsilon^2}$, such that it is connected in $G_p$, yet all the edges of $E(K, K^C)$ do not appear in $G_{p_1}$. By Lemma \ref{l: gap}, \textbf{whp} there is no such set, and therefore \textbf{whp} $G_p$ contains a unique component $L_1$ whose order is at least $\frac{14\log n}{\epsilon^2}$, and all the other components are of order at most $\frac{14\log n}{\epsilon^2}$. 

Furthermore, by Lemma \ref{l: concentration}, \textbf{whp} there are $(1+o(1))y(\epsilon)n$ vertices in $G_p$ in components whose order is at least $\frac{14\log n}{\epsilon^2}$. Since \textbf{whp} there is only one component, $L_1$, whose order in $G_p$ is at least $\frac{14\log n}{\epsilon^2}$, we conclude that \textbf{whp} $|V(L_1)|= (1+o(1))y(\epsilon)n$. \qed

\section{Proofs of Theorems \ref{th: main large degree} and \ref{th: main small degree}}\label{s: variants}
Let us explain how slight modifications of the proof of Theorem \ref{th: main} yield the proofs of Theorems~\ref{th: main large degree} and~\ref{th: main small degree}.
\paragraph{High degree.} We show that we can replace Property \ref{p: ball growth} with the assumption that $d\ge 10\log n$. 

First, note that Lemma \ref{l: components on logarithmic sets} and Corollary \ref{col: components on logarithmic sets} hold, verbatim, if we replace $c'd\log n$ with $d$, under the assumption that $d\ge \frac{10\log n}{\epsilon}$. That is, one can obtain the following lemma.
\begin{lemma}\label{l: components on logarithmic sets large degree}
Let $S\subseteq V$ with $|S|=d$. Then, with probability at least $1-o\left(1/n\right)$ there exists a subset $X\subseteq S$, with $|X|\ge \epsilon^2 d$, such that $X\subseteq V_L(G_p)$.
\end{lemma}

Thus, since the graph is $d$-regular, by the union bound the following is an immediate corollary of Lemma~\ref{l: components on logarithmic sets large degree}.
\begin{lemma}
\textbf{Whp}, every $v\in V$ is at distance one in $G$ from at least $\epsilon^2 d$ vertices in $V_L(G_p)$.
\end{lemma}

The proof of Theorem \ref{th: main large degree} then continues in the same manner as the proof of Theorem \ref{th: main}, where instead of using paths of length $\Omega(\log_d \log n)$ to merge the components, we use paths of length three. As for the case where $d=\Theta(n)$, utilising a `gap' statement similar to Lemma \ref{l: gap}, we have that large components are of order at least $\epsilon^5n$. Our the assumption that $c_1=\omega(d/n)$ gives that there will be $\omega(d)$ edges between $A'$ and $B'$ in $G$, and thus \textbf{whp} there will be a matching of size $\omega(1)$ between $A'$ and $B'$ in $G_{p_2}$. Then, all that is left is to observe that since large components are of order at least $\epsilon^5n$, there is only a constant number of partitions to consider.  \qed

Let us note here that the choice of $\epsilon^5$ is arbitrary and for the sake of readability, and could be replaced with any small constant depending on $\epsilon$. 

\paragraph{Low degree.} We show that one can relax the assumption on $d$ to $d=\omega(1)$, as well as remove the requirement of Property \ref{p: ball growth}, if we ask for a stronger small-set expansion, that is, that sets up to size $(d\log n)^{5\log \log n}$ have almost optimal edge expansion. Indeed, we may strengthen our `gap' statement (Lemma \ref{l: gap}), and show that \textbf{whp} there are no components whose order is between $\Omega(\log n/\epsilon^2)$ and $(d\log n)^{5\log \log n}$. We will also utilise the following lemma, relating our edge expansion assumption to the growth rate of balls.
\begin{lemma}\label{l: edge expansion to vertex expansion}
Let $G$ be a $d$-regular graph on $n$ vertices, with $d=\omega(1)$. Let $ k, r\ge 1$ be integers, and suppose that for every connected $U\subseteq V(G)$ with $|U|\le k$, $e(U,U^C)\ge (1-\epsilon^3)d|U|$. Let $v\in V(G)$. Then, 
\begin{align*}
    |B_G(v, r)|\ge \min\left\{k, \left(\frac{1}{\epsilon^3}\right)^r\right\}.
\end{align*}
\end{lemma}
\begin{proof}
For all $\ell \in [r]$ let $X_{\ell}\coloneqq B_G(v,\ell)$. Since $G$ is $d$-regular, $|X_1| \geq d +1 \geq \min\left\{k, \frac{1}{\epsilon^3}\right\}$. Suppose that $2 \leq \ell <r$ and  $|X_{\ell+1}|\le k$. Then, by our assumption, for each $i \leq \ell+1$ we have that $e(X_{i})\le \frac{\epsilon^3}{2}d|X_{i}|$. On the other hand, since $G$ is $d$-regular,
\begin{align*}
    e(X_{\ell+1})&= d|X_{\ell}| - e(X_{\ell}) \ge (1-\epsilon^3/2)d|X_{\ell}|.
\end{align*}
Hence, $\frac{\epsilon^3}{2}d|X_{\ell_1}|\ge (1-\epsilon^3/2)d|X_{\ell}|$, and in particular $|X_{\ell+1}|\ge \frac{1}{\epsilon^3}|X_{\ell}|$

Thus, we obtain that 
\[
|X_{\ell+1}|\ge \min\left\{k, \frac{1}{\epsilon^3}|X_\ell|\right\},
\]
and therefore $|B_G(v,r)| = |X_r| \ge \min\left\{k, \left(\frac{1}{\epsilon^3}\right)^r\right\}$, as required.
\end{proof}
Therefore, by Lemma \ref{l: edge expansion to vertex expansion}, we have that $|B(v,\log\log n)|\ge \frac{10\log n}{\epsilon}$. Thus, we can argue similarly to before that \textbf{whp} every vertex in $G$ has at least one vertex in $V_L(G_p)$ at distance at most $\log\log n$. We continue assuming this holds deterministically. 

We now turn to show that large components typically merge, as the rest of the proof follows verbatim. We remain with the same notation of $W, A,$ and $B$ as in the proof of Theorem~\ref{th: main}. There are at most $\binom{n/(d\log n)^{5\log \log n}}{|A|/(d\log n)^{5\log \log n}}$ ways to partition $W$ into $A\sqcup B$. Let $A'$ be $A$ together with the set of vertices in $V(G)\setminus (A\cup B)$ which have at least one vertex in $A$ at distance at most $\log\log n$ from them, and let $B'$ be $B$ together with the set of vertices in $V(G)\setminus (A\cup B\cup A')$ which have at least one vertex in $B$ at distance at most $\log\log n$ from them. By the above, $V=A'\sqcup B'$. By Property \ref{p: global expansion}, $e(A',B')\ge c_1|A|$. Very crudely, we can extend these edges to $\left(\frac{1}{d}\right)^{2\log\log n}c_1|A|$ edge (in fact vertex) disjoint paths of length at most $2\log\log n+1$ between $A$ and $B$ in $G$. Hence, by the union bound, the probability there is such a partition without a path between $A$ and $B$ in $G_{p_2}$ is at most
\begin{align*}
    \sum_{|A|=(d\log n)^{5\log\log n}}^{n/2}&\binom{n/(d\log n)^{5\log\log n}}{|A|/(d\log n)^{5\log\log n}}(1-p_2^{2\log\log n+1})^{\left(\frac{1}{d}\right)^{2\log\log n}c_1|A|}\\&\le \sum_{|A|=(d\log n)^{5\log\log n}}^{n/2}\exp\left\{|A|\left(\frac{\log n}{(d\log n)^{5\log\log n}}-\left(\frac{1}{d}\right)^{5\log\log n}\right)\right\}\\
    &\le n\cdot\exp\left\{-\frac{|A|}{2d^{5\log \log n}}\right\}=o(1),
\end{align*}
where we used that $|A|\ge (d\log n)^{5\log\log n}$. \qed

\section{Proof of Theorem \ref{th: construction}}\label{s: construction}
The construction given here is similar in spirit to that given in \cite{DEKK25}. Let $c_1'\coloneqq 3c_1$, let $d'\coloneqq d-c_1'$ and let $t\coloneqq \frac{30c_1n}{d\log(n/d)}$. We further assume that $c_1, d, t$, and $n$ satisfy the needed parity assumptions for what follows.

Let $H$ be a $c_1'$-regular graph on $n$ vertices, such that every $U\subseteq V(H)$ with $|U|\le |V(H)|/2$ satisfies that $|E_H(U,U^C)|\ge \frac{c_1'}{3}\cdot |U|=c_1 |U|$ (\textbf{whp} a random $c_1'$-regular graph on $n$ vertices satisfies this). Since $t=\frac{30c_1n}{d\log(n/d)}=\omega(1)$ (as we assume that $d=o(n)$) and $c_1'$ is a constant, we conclude that there exists an equitable (proper) colouring of $H$ in $t$ colours, $A_1,\ldots, A_t$, with each colour class containing exactly $\frac{d\log(n/d)}{30c_1}$ vertices \cite{HS70}. Form $G$ by placing in $H[A_j]$, for every $j\in[t]$, a copy of an $(n',d',\lambda')$-graph, where $n'=\frac{1}{30c_1}d\log(n/d)$ and $\lambda'\le d^{3/4}$. Since $d=d'+c_1'$, we have that $G$ is a $d$-regular graph on $n$ vertices. Further, for every $U\subseteq V(G)$, $e(U,U^C)\ge |E_H(U,U^C)|\ge c_1|U|$. Moreover, by construction, for every $U\subseteq V(G)$ with $|U|\le \frac{\log(n/d)}{40c_1}$ we have that $|N_G(U)|\ge d|U|$ (as we can consider the partition of $U$ according to $A_i$), and that by the expander mixing lemma~\cite{AC88}, for every $U\subseteq V(G)$ with $|U|\le \frac{\epsilon^3 d\log(n/d)}{100c_1}$, we have that $e(U,U^C)\ge (1-\epsilon^3)d|U|$, and thus $G$ satisfies the assumptions of Theorem~\ref{th: construction}.

Note that, for every $j\in [t]$, the edges between $A_j$ and $V\setminus A_j$ are those in $H$. Let $X$ be the number of sets $A\in\{A_1, \ldots, A_t\}$, such that $E_{H_p}(A, A^C)=\varnothing$. For each fixed $j\in[t]$ we have that $|E_H(A_j, V(H)\setminus A_j)|=c_1'\cdot \frac{d\log(n/d)}{30c_1}=\frac{d\log(n/d)}{10}$. The probability $E_{H_p}(A_j,V(H)\setminus A_j)=\varnothing$ is 
\begin{align*}
    (1-p)^{d\log(n/d)/10}\ge \exp\left\{-\log(n/d)/5\right\}\ge (d/n)^{1/5},    
\end{align*}
Hence, $\mathbb{E}[X]\ge t(d/n)^{1/5}\ge (n/d)^{3/4}$. Now, note that changing one edge can change the value of $X$ by at most two. Hence, by Lemma~\ref{l: azuma},
\begin{align*}
    \mathbb{P}\left[X\le \frac{\mathbb{E}[X]}{2}\right]\le 2\exp\left\{-\frac{\left(\frac{n}{2d}\right)^{6/4}}{2\cdot\frac{c_1'np}{2}\cdot 4}\right\}\le 2\exp\left\{-\Omega\left(\frac{n^{1/2}}{d^{1/2}}\right)\right\}=o(1),
\end{align*}
where we used $d=o(n)$. Thus, \textbf{whp} there are at least $(n/d)^{2/3}$ sets $A\in \{A_1, \ldots, A_t\}$ satisfying $E_{G_p}(A,A^C)=\varnothing$. By \cite{FKM04}, for every fixed $i\in [t]$ \textbf{whp} there exists a component of order at least $\epsilon |A_i|=\frac{\epsilon d\log (n/d)}{30c_1}$ in $G_p[A_i]$. Therefore, \textbf{whp} there are at least two sets $A_i,A_j\in \{A_1,\ldots, A_t\}$, with $i\neq j$, such that $E_{G_p}(A_i,V(G)\setminus A_i)=E_{G_p}(A_j,V(G)\setminus A_j)=\varnothing$ and there exist a components of order at least $\frac{\epsilon d\log (n/d)}{30c_1}$ in $G_p[A_i]$ and similarly a component of order at least $\frac{\epsilon d\log (n/d)}{30c_1}$ in $G_p[A_j]$. As both $A_i$ and $A_j$ are isolated from the rest of the graph in $G_p$, \textbf{whp} $G_p$ has at least two components of order at least $\frac{\epsilon d\log (n/d)}{30c_1}$. \qed

\section{Discussion}\label{s: discussion}
We showed that for a regular graph $G$ of growing degree $d$, some very mild assumption on the edge expansion properties of $G$, and a fairly good control over the expansion of sets up to size $O(d\log n)$, suffices to ensure that $G$ will exhibit the Erd\H{o}s-R\'enyi component phenomenon (ERCP). We further showed that our edge expansion assumption on sets up to size $O(d\log n)$ is fairly tight, in the sense that there are graphs with almost optimal edge expansion of sets up to size $\Omega(d\log n)$ which do not exhibit the ERCP.

As mentioned in the introduction, it was shown by Frieze, Krivelevich and Martin \cite{FKM04} that pseudo-random $(n,d,\lambda)$-graphs, where $\lambda=o(d)$, exhibit the ERCP. The classical results of Ajtai, Koml\'os, and Szemer\'edi~\cite{AKS81} and of Bollob\'as, Kohayakawa, and \L{}uczak~\cite{BKL92} show that the hypercube $Q^d$ exhibits the ERCP as well. The proofs of these two results are quite different, with the first relying on quite a tight control on \textit{edge-distribution} of the graph (through the expander mixing lemma), and the latter on the \textit{product structure} of the hypercube (together with Harper's edge isoperimetric inequality). Here, we demonstrated that the ERCP can be determined solely from the expansion properties of the graph, thereby providing a unified approach to this natural question.

As is evident from Theorems \ref{th: main}, \ref{th: main large degree} and \ref{th: main small degree}, there is an intrinsic connection between the `global' assumption on edge expansion (Property \ref{p: global expansion}), and the `local' assumption on the edge expansion of small sets (Property \ref{p: small set expansion}): the stronger the assumption on the global expansion is, the weaker the assumption on the expansion of small sets can be (and vice versa). As demonstrated by Theorem \ref{th: construction}, this connection is not merely a by-product of our proof technique and can be seen to be tight, at least in a qualitative sense. It would be interesting to obtain a qualitative `tight' understanding of this connection.

Finally, while Theorem \ref{th: construction} shows that the assumption of edge expansion of small sets (Property \ref{p: small set expansion}) in Theorem \ref{th: main} is tight, it remains an open question whether the assumption of vertex expansion (Property \ref{p: ball growth}) is indeed necessary. Recall that Property \ref{p: ball growth} was used only to establish that a ball of radius $\log_d\log n+1$ contains $\Omega(d\log n)$ vertices. Moreover, by Theorem \ref{th: main large degree}, this assumption can be removed for graphs whose degree is at least $10\log n$, and by Theorem \ref{th: main small degree}, asking for nearly optimal edge expansion for sets larger-sized sets allows one to remove this assumption as well. 
\begin{question}
    Let $d=\omega(1)$, let $\epsilon>0$ be a sufficiently small constant, and let $p=\frac{1+\epsilon}{d}$. Is there a $d$-regular graph on $n$ vertices, satisfying Properties \ref{p: global expansion} and \ref{p: small set expansion}, for which \textbf{whp} the second largest component in $G_p$ is of order $\omega(\log n)$?
\end{question}

\paragraph*{Acknowledgement} We thank the anonymous referee for his/her careful reading, and helpful suggestions and comments, which helped improve the quality of this paper.

\bibliographystyle{abbrv}
\bibliography{perc} 
\end{document}